\documentclass[11pt]{article}
\usepackage{epigamath}

\usepackage[notext]{kpfonts}
\usepackage{baskervald}

\setpapertype{A4}

\usepackage[english]{babel}

\usepackage[dvips]{graphicx}     % Use this instead with dvips

\title{Double spinor Calabi-Yau varieties}
\titlemark{Double spinor Calabi-Yau varieties}
\author{\vspace{0cm} Laurent Manivel}
\authoraddresses{
\authordata{Laurent Manivel}{\firstname{Laurent} \lastname{Manivel}\\
\institution{Institut de Math\'ematiques de Toulouse, UMR 5219, Universit\'e de Toulouse, CNRS, UPS IMT F-31062 Toulouse Cedex 9, France}\\
\email{manivel{@}math.cnrs.fr}}
}
\authormark{L. Manivel}
\date{\vspace{-5ex}} % Empty date or tweak it according to your needs
\journal{\'Epijournal de G\'eom\'etrie Alg\'ebrique} % Epijournal name
\acceptation{Received by the Editors on September 28, 2017, and in final form
on December 4, 2018. \\ Accepted on February 5, 2019.}

\newcommand{\articlereference}{Volume 3 (2019), Article Nr. 2}

\usepackage[all]{xy}

\allowdisplaybreaks

\newtheorem{theorem}{Theorem}[section]
\newtheorem{lemma}[theorem]{Lemma}
\newtheorem{proposition}[theorem]{Proposition}
\newtheorem{corollary}[theorem]{Corollary}

\def\PP{\mathbf P}
\def\CC{\mathbb{C}}
\def\ZZ{\mathbb{Z}}

\def\cO{{\mathcal O}}
\def\cHS{\mathcal{HS}}
\def\ra{\rightarrow}\def\lra{\longrightarrow}

\def\af1{\mathbf{aff}_1}

\def\cS{\mathcal S}\def\cO{\mathcal O}\def\cI{\mathcal I}
\def\ra{\rightarrow}
\def\lra{\longrightarrow}

\DeclareMathOperator{\Spin}{Spin}
\DeclareMathOperator{\GL}{GL}
\DeclareMathOperator{\SL}{SL}
\DeclareMathOperator{\PGL}{PGL}
\DeclareMathOperator{\SO}{SO}
\DeclareMathOperator{\PSO}{PSO}
\DeclareMathOperator{\Aut}{Aut}
\DeclareMathOperator{\Pic}{Pic}

\begin{document}

%%%%%%%%%%%%%%%%%%%%%%%%%%%%%%%
% Add the title to the document
%%%%%%%%%%%%%%%%%%%%%%%%%%%%%%%

\maketitle

%\contribution{}

%%%%%%%%%%%%%%%%%%%%%
% Dedication (if any)
%%%%%%%%%%%%%%%%%%%%%
%\dedication{}

%%%%%%%%%%%%%%%%%%%%%%%%%%%%%%%%%%%%%%%%%%%%%%%%%%%%%%%%%%
% Add abstract, Keywords, MSC classification (recommended)
% Never remove prelims section, make it rather empty
%%%%%%%%%%%%%%%%%%%%%%%%%%%%%%%%%%%%%%%%%%%%%%%%%%%%%%%%%%
\begin{prelims}

%\vspace{-0.55cm}

\def\abstractname{Abstract}
\abstract{Consider the ten-dimensional spinor variety $\cS$ in a projectivized half-spin representation of $\Spin_{10}$.
This variety is projectively  isomorphic to its projective dual $\cS^\vee$ in the 
dual projective space. 
The intersection $X=\cS_1\cap \cS_2$ of two general translates of $\cS$ is a smooth Calabi-Yau
fivefold, as well as the intersection of their duals $Y=\cS_1^\vee\cap \cS_2^\vee$. We prove that 
although $X$ and $Y$ are not birationally equivalent, they are deformation equivalent, Hodge equivalent, derived equivalent and L-equivalent.}

\keywords{Spinor variety, Calabi-Yau variety, derived equivalence,
birational equivalence, projective duality}

\MSCclass{14C34; 14F05; 14J32; 14L35}

\vspace{0.15cm}

\languagesection{Fran\c{c}ais}{%

%\vspace{-0.05cm}
\textbf{Titre. Vari\'et\'es de spineurs doubles de Calabi-Yau} \commentskip \textbf{R\'esum\'e.} Soit $\mathcal{S}$ la vari\'et\'e spinorielle de dimension 10 dans le projectivis\'e d'une repr\'esentation semi-spin de $\Spin_{10}$. Cette vari\'et\'e est projectivement isomorphe \`a son dual projectif $\mathcal{S}^\vee$ dans l'espace projectif dual. L'intersection $X=\mathcal{S}_1\cap\mathcal{S}_2$ de deux translat\'es g\'en\'eraux de $\mathcal{S}$ est une vari\'et\'e lisse de Calabi-Yau de dimension 5, de m\^eme que l'intersection $Y=\mathcal{S}_1^\vee\cap \mathcal{S}_2^\vee$ de leurs duaux. Bien que $X$ et $Y$ ne soient pas birationnellement \'equivalentes, nous montrons qu'elles sont
\'equivalentes par d\'eformations, par \'equivalence de Hodge, par \'equivalence d\'eriv\'ee et qu'elles sont \'egalement L-\'equivalentes.}

\end{prelims}

%%%%%%%%%%%%%%%%%%%%%
% Content begins here
%%%%%%%%%%%%%%%%%%%%%

\newpage

% Add table of contents (optional)
\setcounter{tocdepth}{1} \tableofcontents

\section{Introduction} 

There has been a lot of recent interest in the relations, for pairs of Calabi-Yau threefolds, between derived equivalence, 
Hodge equivalence and birationality. The  Pfaffian-Grassmannian equivalence provided pairs of non birational 
Calabi-Yau threefolds which are derived equivalent, but with distinct topologies. Recently, examples were found of 
non birational Calabi-Yau threefolds which are derived equivalent, and also deformation and Hodge equivalent. They 
are constructed from the six-dimensional Grassmannian $G(2,V)\subset \PP(\wedge^2V)$, where $V$ denotes a five-dimensional
complex vector space. This Grassmannian is well-know to be projectively self dual, more precisely its projective dual is 
$G(2,V^\vee)\subset \PP(\wedge^2V^\vee)$, where $V^\vee$ denotes the dual space to $V$. 
Let $Gr_1$ and $Gr_2$ be two $\PGL(\wedge^2V)$-translates of $G(2,V)$ in $\PP(\wedge^2V)$, and 
suppose they intersect transversely. Then $Gr_1^\vee$ and $Gr_2^\vee$ also intersect transversely in $\PP(\wedge^2V^\vee)$.
Moreover
$$X=Gr_1\cap Gr_2 \qquad \mathrm{and} \qquad Y=Gr_1^\vee\cap Gr_2^\vee$$
are smooth Calabi-Yau threefolds with the required properties. This was established independently in \cite{or} and \cite{bcp}, 
to which we refer for more details on the general background. We should note however that those remarkable threefolds 
had already appeared in the litterature, see \cite{gp,kan,kap}. 

The purpose of this note is to show that the very same phenomena occur if we replace the Grassmannian $G(2,V)\subset \PP(\wedge^2V)$
by the ten-dimensional spinor variety $\cS\subset \PP\Delta$, where $\Delta$ denotes one of the two half-spin representations 
of $\Spin_{10}$. These representations have dimension $16$. Recall that $\cS$ parametrizes one of the two families of maximal isotropic spaces in a ten-dimensional quadratic 
vector space. The projective dual of  $\cS\subset \PP\Delta$ is the other such family $\cS^\vee\subset \PP\Delta^\vee$, which is 
projectively equivalent to $\cS$.  Let $\cS_1$ and $\cS_2$ be two $\PGL(\Delta)$-translates of $\cS$ in $\PP\Delta$, and 
suppose that they intersect transversely. Then $\cS_1^\vee$ and $\cS_2^\vee$ also intersect transversely in $\PP\Delta^\vee$.
Moreover
$$X=\cS_1\cap \cS_2 \qquad \mathrm{and} \qquad Y=\cS_1^\vee\cap \cS_2^\vee$$
are smooth Calabi-Yau fivefolds of Picard number one which are deformation equivalent, derived\footnote{D-equivalent in the sequel.} equivalent (Proposition \ref{dereq}), 
Hodge equivalent (Corollary \ref{hodgeq}), but not birationally equivalent (Proposition \ref{nbireq}). Also, 
the difference of their classes in the 
Grothendieck ring of varieties is annihilated by a power of the class of the affine line (Proposition \ref{relgro}): in the 
terminology of \cite{ks}, $X$ and $Y$ are L-equivalent. This confirms their conjecture, which also
appears as a question in \cite{IMOU}, that (at least for simply connected projective 
varieties) D-equivalence should imply L-equivalence.

From the point of view of mirror symmetry, $X$ and $Y$ being D-equivalent should have the same mirror, and in this respect they form a double mirror. 
Of course, from the projective point of view they are 
also (projective) mirrors one of the other. 
\smallskip

The close connection between the Grassmannian $G(2,5)\subset\PP^9$ and the spinor variety $\cS\subset\PP^{15}$ is classical, 
and manifests itself at different levels.
\begin{enumerate}
\item They are the only two Hartshorne varieties among the rational homogeneous spaces, if we define a Hartshorne 
variety to be a smooth variety $Z\subset\PP^N$ of dimension $n=\frac{2}{3}N$ which is not a complete intersection 
(recall that Hartshorne's conjecture predicts that $n>\frac{2}{3}N$ is impossible) \cite[Corollary 2.16]{zak}.
\item They are both prime Fano manifolds of index $\iota=\frac{N+1}{2}$, while their topological Euler characteristic 
is also equal to $N+1$; this allows their derived categories to admit rectangular Lefschetz decompositions of length 
$\iota$, based on the similar pairs $\langle\cO_Z,U^\vee\rangle$, where $U$ denotes their tautological bundle \cite{kuz}. 
\item The Grassmannian $G(2,5)\subset\PP^9$ can be obtained from  $\cS\subset\PP^{15}$ as parametrizing the lines in $\cS$ 
through some given point \cite{lm2}; conversely, $\cS\subset\PP^{15}$ can be reconstructed from $G(2,5)\subset\PP^9$ by a simple 
quadratic birational map defined in terms of the quadratic equations of the Grassmannian \cite{lm1}.
\end{enumerate}

That the story should be more or less the same for our double spinor varieties, as for the double Grassmannians, is therefore not a big surprise. We thought it would nevertheless be useful to check that everything 
was going through as expected. For that we essentially followed the ideas of \cite{bcp} and \cite{or}, to which this note is of 
course heavily indebted.

\medskip\noindent
{\it Acknowledgements}. We thank M. Brion, A. Kanazawa, A. Kuznetsov, J.C. Ottem, A. Perry, J. Rennemo for their comments and hints. 

\section{Spinor varieties}

\subsection{Pure spinors}

We start with some basic facts about spin representations. See \cite{manspin} for more details, and references therein. Our base field all over the paper will be the field of complex numbers. 
For convenience we will restrict to even dimensions, so we let $V=V_{2n}$ be a complex vector space of dimension $2n$, 
endowed with a non degenerate quadratic form. The variety of isotropic $n$-dimensional subspaces of $V$, considered as 
a subvariety of the Grassmannian $G(n,2n)$, has two 
connected components
$$\cS_+=OG_+(n,2n)\quad \mathrm{and}\quad \cS_-=OG_-(n,2n),$$
called the spinor varieties, or varieties of pure spinors. 
Moreover the Pl\"ucker line bundle restricted to $\cS_+$ or $\cS_-$ has a square root $L$, which is still very 
ample and embeds the spinor varieties into the projectivizations of the two half-spin representations of $\Spin_{2n}$, the 
simply connected double cover of $\SO_{2n}$. We denote the half-spin representations by $\Delta_+$ and $\Delta_-$, in such 
a way that 
$$\cS_+\subset\PP\Delta_+ \qquad \mathrm{and} \qquad \cS_-\subset\PP\Delta_-.$$
Since the two half-spin representations can be exchanged by an outer automorphism of $\Spin_{2n}$, 
these two embeddings are projectively equivalent. Note that the spinor varieties have dimension $n(n-1)/2$, while
the half-spin representations have dimension $2^{n-1}$. The half-spin representations are self dual when $n$ is even, 
and dual one to the other when $n$ is odd. 

It follows from the usual Bruhat decomposition that 
the Chow ring of $\cS_\pm$ is free, and the dimension of its $k$-dimensional component is equal to the number 
of strict partitions of $k$ with parts smaller than $n$. In particular the Picard group has rank one, and $L$ is a generator; we therefore denote $L=\cO_{\cS_\pm}(1)$. We also let 
$U$ be the rank $n$ vector bundle obtained by restricting the tautological bundle of $G(n,2n)$. Then the tangent 
bundle to $\cS_\pm$ is isomorphic to $\wedge^2U^\vee$. Since $\det (U^\vee)=L^2$, 
this implies that $\cS_\pm$ is a prime Fano manifold of index $2n-2$.

\subsection{The ten dimensional spinor variety}

From now on we specialize to $n=5$, and we simply denote $\cS\subset\PP\Delta$ one of the 
spinor varieties. This is a ten-dimensional prime Fano manifold of index eight, embedded in 
codimension five. This case is specific for several reasons. First, it admits a very simple 
rational parametrization.

\begin{proposition}\label{parametrization}
Let $E\subset V_{10}$ be a maximal isotropic subspace, that defines a point of $\cS$. 
Then $\cS\subset\PP \Delta$ is projectively isomorphic to the image of the rational map
from $\wedge^2E^\vee$ to $\PP(\CC\oplus \wedge^2E^\vee\oplus \wedge^4E^\vee)$ that sends $\omega$
to $[1,\omega,\omega\wedge\omega]$. 
\end{proposition} 

Note that this rational map is $\GL(E)$-equivariant, and that $\GL(E)$ is a Levi factor of the 
parabolic subgroup of $\Spin_{10}$ that stabilizes the base point of $\cS$ defined by $E$.

\begin{proof} 
Let us fix another isotropic subspace $F$ of $V_{10}$, transverse to $E$. The quadratic form
identifies $F$ with the dual of $E$. A general point of $\cS$ corresponds to a subspace 
of $V_{10}$ defined by the graph of a map $\omega$ from $E$ to $F\simeq E^\vee$; the isotropy condition 
translates to the skew-symmetry of this map. The embedding to $\PP \Delta$ is then  given 
by the Pfaffians of $\omega$ of any even size: that is $1$ in size $0$, $\omega$ itself in 
size $2$, $\omega\wedge\omega$ in size four. See \cite[Section 2.3]{manspin} or \cite[Theorem 2.4]{lm1} for more details. 
\hfill $\Box$
\end{proof}

A useful consequence is the following. 
In the preceding description of $\cS$, its projectivized tangent space at the base point 
is $\PP(\CC\oplus \wedge^2E^\vee)$ and the normal space identifies with the remaining factor
$\wedge^4E^\vee$. This identifies the representation of $\GL(E)$ that defines the normal
bundle as a homogeneous bundle on $\cS$ and we deduce our next statement. 
Recall that $U$ denotes the tautological rank $5$ 
vector bundle on $\cS$, and that $\det(U)=\cO_\cS(-2)$, where  $\cO_\cS(1)$ is the positive 
generator of $\Pic(\cS)$. 

\begin{proposition}\label{normal}
The  normal bundle to $\cS$ in $\PP\Delta$ is $\wedge^4U^\vee\simeq U(2)$. 
\end{proposition} 

Another nice property that $\cS$ shares with $G(2,5)$ is that 
its complement is homogeneous.

\begin{proposition}\label{prehom}
The action of $\Spin_{10}$ on $\PP\Delta -\cS$ is transitive.
\end{proposition}

In particular $\Delta$ admits a prehomogeneous action, not of  $\Spin_{10}$, but of $\GL(1)\times  \Spin_{10}$. 
This is discussed on page 121 of \cite{sk}.
In fact, a much stronger result is proved in \cite[Proposition 2]{igusa}: over any field of characteristic different from two, there are only two orbits of non zero spinors.

\medskip
The next important property also shared with $G(2,5)$ is that the 
quadratic equations allow to recover the natural representation. 

\begin{proposition}\label{quad}
The quadratic equations of $\cS$ are parametrized by $V_{10}$. 
\end{proposition} 

These equations can be described in terms of the Clifford product, that defines a 
map $V_{10}\otimes \Delta_\pm\ra \Delta_\mp = \Delta_\pm^\vee$. In more down to earth 
terms, we can use the decomposition $\Delta = \CC\oplus \wedge^2E^\vee\oplus \wedge^4E^\vee$
introduced in Proposition \ref{parametrization}: a point $[\omega_0, \omega_2, \omega_4]$ 
belongs to $\cS$, as is readily verified, if and only if
$$\omega_0\omega_4=\omega_2\wedge\omega_2, \quad \mathrm{and}\quad \omega_4*\omega_2=0.$$
The first equation is in $\wedge^4E^\vee\simeq E$. In the second equation, we used this isomorphism
in order to identify $\omega_4$ with an element of $E$; then contracting with $\omega_2$ gives an
element of $E^\vee$ that we denoted $\omega_4*\omega_2$. We thus get, as expected, quadratic equations parametrized by $E\oplus E^\vee=V_{10}$. And this is necessarily an identification 
as $\Spin_{10}$-modules since $V_{10}$ is its only nontrivial ten dimensional representation.

\medskip
As any equivariantly embedded rational homogeneous variety, the spinor variety is cut out by 
quadrics. Moreover the spinor variety $\cS$ has a beautiful self-dual minimal 
resolution (necessarily equivariant), which appears in \cite[5.1]{weyman}: 
\begin{equation}\label{resolution}
0\ra \cO(-8)\ra V_{10}(-6)\ra \Delta^\vee(-5)\ra \Delta(-3)\ra V_{10}(-2)\ra\cO\ra \cO_\cS\ra 0.
\end{equation}
For future use let us compute the Hilbert polynomial of $\cS\subset\PP\Delta$. 
Note that $H^0(\cO_S(k))$ is, by the Borel-Weil theorem, the irreducible $\Spin_{10}$-module
of highest weight $k\omega_5$ (where $\omega_5$ is the fundamental weight corresponding to the 
half-spin representation $\Delta^\vee$). Its dimension can thus be computed by a direct application 
of the Weyl dimension formula, and we get 
$$H_\cS(k)=\frac{1}{2^63^35^27}(k+1)(k+2)(k+3)^2(k+4)^2(k+5)^2(k+6)(k+7).$$
In particular, as is well-known, $\cS$ has degree $12$.
Finally the Poincar\'e polynomial is also easy to compute; since the Betti numbers are
given by numbers of strict partitions, as we already mentioned, we readily get that
$$P_{\cS}(t)=(1+t^3)(1+t+t^2+t^3+t^4+t^5+t^6+t^7).$$

\subsection{Self duality}

Our next statement is a well-known direct consequence of Proposition \ref{prehom}:

\begin{corollary}
The spinor variety $\cS\subset\PP\Delta$ is projectively self dual. 
\end{corollary}

To be more precise, the dual variety of the spinor variety $\cS\subset\PP\Delta$ is the other 
spinor variety $\cS^\vee\subset\PP\Delta^\vee$, in the other half-spin representation. 

Note for future use that the self-duality of $\cS$ is preserved at the categorical level, in the sense that $\cS\subset\PP\Delta$ and 
$\cS^\vee\subset\PP\Delta^\vee$ are homologically projectively dual    \cite[Section 6.2]{kuz}. As already mentioned
in the introduction, the derived category 
of coherent sheaves on the spinor variety $\cS$ has a specially nice rectangular Lefschetz decomposition, defined by 
eight translates of the exceptional pair $\langle\cO_S,U^\vee\rangle$.

\smallskip
Another consequence of Proposition \ref{prehom} is that, up to the group action, 
there are only two kinds, up to projective equivalence, of hyperplane sections of $\cS$: the
smooth and the singular ones. Let us briefly describe their geometries. 

\begin{proposition}\label{hssing}
A singular hyperplane section $\cHS_{sing}$ of $\cS$ is singular along a projective space of dimension four.
Moreover $\cHS_{sing}$ admits a cell decomposition and its Poincar\'e polynomial is
$$P_{\cHS_{sing}}(t)=1+t+t^2+2t^3+2t^4+2t^5+2t^6+2t^7+t^8+t^9.$$
\end{proposition}

\begin{proof} Recall that we may consider $\cS$ and $\cS^\vee$ as the two families of maximal isotropic 
subspaces of $V_{10}$. Moreover, if $E$ and $F$ are two maximal isotropic spaces, they belong to the same family if and only if their intersection has odd dimension. Given a point of $\cS$, that we identify, with some abuse, to such an isotropic space $E$, the set of hyperplanes tangent to 
$\cS$ at $E$ defines a subvariety of $\cS^\vee$.

\begin{lemma}\label{tangency}
A\,point $F\in\cS^\vee$\,defines a hyperplane in $\PP\Delta$\,which is tangent to $\cS$\,at $E$,\,if\,and\,only\,if $\dim(E\cap F)=4$.
\end{lemma}

\begin{proof}[Proof of Lemma \ref{tangency}] It is a consequence of Witt's theorem that the action of $\Spin(V_{10})$ on 
$\cS\times \cS^\vee$ has exactly three orbits, characterized by the three possible values 
for the dimension of the intersection (recall that this dimension must be even). 
Therefore the stabilizer of $E$ in $\Spin(V_{10})$ has also three orbits 
in $\cS^\vee$, defined by the three possible values for the dimension of the intersection with $E$.
Geometrically, they have to correspond to the three possible positions with respect to $E$, 
of a hyperplane defined by a point $F\in \cS^\vee$: tangent to $\cS$ at $E$, containing $E$ but
not tangent, or not containing $E$. This implies the claim. 
\hfill $\Box$
\end{proof}

Since an isotropic space $E\in \cS$ such that $\dim(E\cap F)=4$ is uniquely determined by $E\cap F$, the hyperplane defined
by $F$ is tangent to $\cS$ along a subvariety of $\cS$ isomorphic to $\PP F^\vee$. 

For the last assertions, note that a singular hyperplane section of $\cS$ is just a Schubert divisor. By general results 
on the Bruhat decomposition, we know that its complement in $\cS$ is precisely the big cell. So $\cHS_{sing}$ has a cell decomposition given by all the cells of $\cS$ except the big one, and
$P_{\cHS_{sing}}(t)=P_{\cS}(t)-t^ {10}$.
\hfill $\Box$
\end{proof}

\begin{proposition}\label{hsreg}
A smooth hyperplane section $\cHS_{reg}$ of $\cS$ admits a quasi-homo\-ge\-ne\-ous action of its automorphism group,
which is isomorphic to $(\Spin_7\times {\mathbb G}_m)\ltimes {\mathbb G}_a^8$. 
Moreover $\cHS_{reg}$ admits a cell decomposition and its Poincar\'e polynomial is
$$P_{\cHS_{reg}}(t)=1+t+t^2+2t^3+2t^4+2t^5+2t^6+t^7+t^8+t^9.$$
\end{proposition}

\begin{proof} For the first statement, we refer to \cite[Proposition 3.9]{fw}. As observed in 
\cite{fw},  $\cHS_{reg}$ coincides with the horospherical variety that appears as case $4$ of   
\cite[Theorem 1.7]{pas}. In particular, being horospherical it admits an algebraic
cell decomposition. 
Finally the Betti numbers are given by the Lefschetz hyperplane theorem.
\hfill $\Box$
\end{proof}

\section{Double spinor varieties}

In this section we introduce our main objects of interest, the {\it double spinor varieties} 
$$X=\cS_1\cap \cS_2,$$
where $\cS_1=g_1\cS$ and $\cS_2=g_2\cS$ are translates of $\cS$ by $g_1, g_2\in \PGL(\Delta).$
Up to projective equivalence, we can of course suppose that $X=\cS\cap g\cS$ for  $g\in \PGL(\Delta).$

By the Eagon-Northcott generic perfection theorem \cite[Theorem 3.5]{bv}, 
the resolution (\ref{resolution}) gives a free resolution of $\cO_X$ as an $\cO_{\cS_1}$-module: 
\begin{equation}\label{resolutionX}
0\ra \cO_{\cS_1}(-8)\ra V_{10}(-6)\ra \Delta^\vee(-5)\ra \Delta(-3)\ra V_{10}(-2)\ra\cO_{\cS_1}\ra \cO_X\ra 0.
\end{equation}

\subsection{Local completeness}

Let $G=\PGL(\Delta)$, with its subgroup $H=\Aut(\cS)\simeq \PSO_{10}$ (as follows from \cite{demazure}). 
The family of double spinor varieties is
by definition the image of a rational map 
$$\xi : G/H\times G/H \dashrightarrow Hilb(\PP\Delta),$$
where $\xi (g_1,g_2)=g_1\cS\cap g_2\cS$.
Moreover the diagonal left action of $G$ is by projective equivalence, hence factors out when we 
consider local deformations of a given $X$. At the global level, the quotient $[(G/H\times G/H)/G]$
should be thought of as the moduli stack of double spinor varieties. One could reproduce the analysis of the
similar stack made in \cite{bcp} for the double Grassmannians, but we will not do that. We will only check the 
local completeness of our family. 

\begin{proposition}\label{complete}
The family of smooth double spinor varieties is locally complete. 
\end{proposition}

\begin{proof}
We first observe that $H^1(X,T\cS_{1|X})=0$. Because of (\ref{resolutionX}), 
this is a direct consequence of the vanishing of 
$H^1(\cS_1,T\cS_1)$, of $H^q(\cS_1,T\cS_1(-k))$ for $0<k<8$ and $q>1$, and of 
$H^6(\cS_1,T\cS_1(-8))$, which are all consequences of Bott's theorem. (Alternatively, 
the vanishing of  $H^q(\cS_1,T\cS_1(-k))$
follows from the Kodaira-Nakano vanishing theorem,
since this group is Serre dual to $H^{10-q}(\cS_1,\Omega_{\cS_1}(k-8))$.)
Hence the map 
$$H^0(X,N_{X/\cS_1})\simeq H^0(X,N_{\cS_2/\PP|X})\lra H^1(X,TX)$$
is surjective. Here we abbreviated $\PP\Delta$ by $\PP$. What remains to prove is that the composition 
$$H^0(\PP,T\PP)\stackrel{r}{\lra} H^0(X,T\PP_{|X})\stackrel{s}{\lra} H^0(X,N_{\cS_2/\PP|X})$$
is also surjective. In order to prove that $r$ is surjective, it is convenient to use the 
Euler exact sequence on $\PP$ and its restriction to $X$; from (\ref{resolutionX}) we easily get 
that $H^1(X,\cO_X)=0$ and $H^0(\PP, \cO_\PP(1))\simeq H^0(X,\cO_X(1))$, and the surjectivity of $r$
readily follows. 
The surjectivity of $s$ follows from the vanishing of $H^1(X,T\cS_{2|X})$, which we already verified.
\hfill $\Box$
\end{proof}

\medskip
The following observation (already made in \cite[Proposition 4.7]{kap} for the double 
Grassmannian varieties) will be useful: when $g$ goes to identity,  $X=\cS\cap g\cS$ deforms 
smoothly to the zero locus in $\cS$ of a global section 
of its normal bundle $\wedge^4U^\vee=U(2)$. In our situation the discussion of 
\cite[section 5]{or} applies almost verbatim. 

The normal bundle is generated by global sections, being homogeneous 
and irreducible, with $$H^0(\cS,\wedge^4U^\vee)=\wedge^4V_{10}^\vee=V_{\omega_4+\omega_5}$$
by the Borel-Weil theorem. So the zero locus of a general section, which we call a normal degeneration, 
is a smooth Calabi-Yau fivefold which is  deformation equivalent to the smooth double spinor varieties; in particular the family of 
those zero-loci  is not locally complete, something that seems to be quite exceptional.

\paragraph{\bf Remark.}
Note that other kinds of degenerations of double Grassmannians, 
this time singular, were considered in \cite{gp}: typically, they are joins of 
two elliptic quintics (which are linear sections of the Grassmannian) in two disjoint $\PP^4$'s in $\PP^9$. 
Such degenerations were studied in connection with the Horrocks-Mumford vector bundle, in order to describe the 
moduli space of $(1,10)$-polarized abelian surfaces. It would certainly be interesting to study the similar 
story in our setting. The analogous singular degenerations are of course the joins of two K3 surfaces of degree $12$ 
(which are linear sections of the spinor variety) in two disjoint $\PP^7$'s in $\PP^{15}$.

\subsection{Invariants}

\begin{proposition}\label{cy}
Any smooth double spinor variety $X=\cS_1\cap \cS_2$ (of the expected dimension) 
is a Calabi-Yau fivefold. 
Moreover:
\begin{enumerate}
\item $\Pic(X)=\ZZ\cO_X(1)$, and $H^p(X,\Omega^p_X)=\CC$ for $0\le p\le 5$; 
\item $H^q(X,\Omega^p_X)=0$ for $p\ne q$ and $p+q\ne 5$;
\item $H^5(X,\ZZ)$ is torsion free. 
\end{enumerate}
\end{proposition}

\begin{proof}
The hypothesis that $X$ is smooth of the expected dimension is equivalent to the fact 
that $\cS_1$ and $\cS_2$ meet transversely. Then $X$ 
has dimension five. Suppose to simplify notations that $\cS_1=\cS$. 

Since $\omega_S=\cO_S(-8)$, the line bundles $\cO_\cS(-k)$ are acyclic for $0<k<8$. 
Moreover $h^q(\cO_\cS(-8))=0$ for $q\le 9$. Then (\ref{resolutionX}) yields that
$h^0(\cO_X)=1$, so that $X$ is connected. 
Moreover, 
the relative dualizing sheaf 
$$\omega_{X/\cS}=\det N_{X/\cS}= \det N_{\cS_2/\PP|X}=\cO_X(8),$$
and since $\omega_{\cS}=\cO_\cS(-8)$, 
we conclude that $X$ has trivial canonical bundle.

As was done in \cite[Lemma 3.3]{or}, we now apply  \cite[Corollary b)]{sommese}
to $A=\cS_1$ and $B=\cS_2$ in $\PP^{15}$: we get that the relative 
homotopy groups $\pi_i(\cS_1,X)=0$  for $i\le 5$. In particular $X$ is simply connected, and by the Bogomolov decomposition theorem, it is Calabi-Yau.  Moreover, after passing from homotopy to homology, we deduce that $H_i(\cS_1,X,\ZZ)=0$  for $i\le 5$. Since 
the cohomology of $\cS_1$ is pure and torsion free,  this implies the remaining assertions.
\hfill $\Box$
\end{proof}

Knowing the Hilbert polynomial of $\cS$, we also deduce that the Hilbert polynomial of $X$ is
$$H_X(k)=\frac{2}{5}k(k^2+1)(3k^2+17).$$

\begin{proposition}
The non zero Hodge numbers of a smooth double spinor variety are $h^{p,p}=1$ for $0\le p\le 5$, and 
$$h^{0,5}=h^{5,0}=1, \quad h^{1,4}=h^{4,1}=165, \quad h^{2,3}=h^{3,2}= 7708.$$ 
\end{proposition}

\begin{proof}
The missing Hodge number are  $h^4(\Omega_{X})=1+\chi(\Omega_{X})$ and $h^3(\Omega^2_{X})=1-\chi(\Omega^2_{X})$. 
In order to compute them, 
we may suppose that $X$ is a normal degeneration of a double spinor variety, that is, the 
zero locus of a general section of $\wedge^4U^\vee=U(2)$ on $\cS$. 
Then the bundle of forms on $X$ is resolved by the conormal sequence, and the bundle of two-forms  by its
skew-symmetric square, that is 
$$0\ra S^2U_X^\vee(-4)\ra \wedge^2U_X\otimes U_X^\vee(-2)\ra \wedge^2(\wedge^2U_X) \ra\Omega^2_{X}\ra 0.$$
This allows to compute $\chi(\Omega_{X})$  as  $\chi(\wedge^2U_X)-\chi(U_X^\vee(-2))$, and 
$\chi(\Omega^2_{X})$
as $\chi(\wedge^2(\wedge^2U_X))-\chi(\wedge^2U_X\otimes U_X^\vee(-2))+\chi(S^2U_X^\vee(-4))$. Using the 
Koszul complex, this reduces 
once again to computations on the spinor variety. Finally, on the latter we can use the Borel-Weil-Bott theorem
(whose concrete application is illustrated in the proof of Lemma \ref{van} below). 
\hfill $\Box$
\end{proof}

\subsection{Uniqueness} 

In this section we prove that the only translates of $\cS$ that contain $X=\cS\cap g\cS$ are $\cS$ itself, and $g\cS$. 
In particular there is a unique way to represent $X$ as an intersection of two translates of the spinor variety. 
We follow the approach of \cite{bcp}. 

\begin{proposition}
Let $N$ denote the normal bundle to $\cS$ in $\PP\Delta$. Then the restriction of $N$ to $X$ is slope stable.  
\end{proposition}

\begin{proof} 
Recall that we denoted by $U$ the rank five tautological bundle on $\cS$. By Proposition 
\ref{normal} there is an isomorphism 
$N\simeq \wedge^4U^\vee=U(2)$, so we just need to prove that $U_X^\vee$ is stable. 
Since the Picard group of $X$ is cyclic by Proposition \ref{cy}, we can apply Hoppe's criterion
\cite[Proposition 1]{jardim}, following which it is enough to check that 
$$H^0(X, U_X^\vee(-1))=H^0(X,\wedge^2U_X^\vee(-1))=H^0(X,\wedge^3U_X^\vee(-2))
=H^0(X,\wedge^4U_X^\vee(-2))=0.$$
This easily follows from the resolution (\ref{resolutionX}) and the following statement. 
\hfill $\Box$
\end{proof}

\begin{lemma}\label{van}
Suppose that $1\le e\le 4$, $0\le q\le 5$ and $t>0$. Then $H^q(\cS,\wedge^eU^\vee(-t))=0$, except for the following cohomology groups:
\begin{enumerate}
\item  $H^1(\cS,\wedge^3U^\vee(-2))=\CC$,
\item  $H^0(\cS,\wedge^4U^\vee(-1))=\Delta$.
\end{enumerate}
\end{lemma}

\begin{proof} This is a straightforward application of the Borel-Weil-Bott theorem 
(see e.g. \cite[Theorem 2.1]{bor} and references therein). 
Let us prove the first identity, to explain how this theorem applies in our setting. 
The root system $D_5$ can be described in terms of a lattice 
with orthonormal basis $\epsilon_1, \ldots , \epsilon_5$. The simple roots of $D_5$ can 
be chosen to be 
$$\alpha_1=\epsilon_1-\epsilon_2, \; \alpha_2=\epsilon_2-\epsilon_3, \; \alpha_3=\epsilon_3-\epsilon_4, \; \alpha_4=\epsilon_4-\epsilon_5, \; \alpha_5=\epsilon_4+\epsilon_5.$$
The fundamental weights are then $\omega_1=\epsilon_1$, $\omega_2=\epsilon_1+\epsilon_2$,  $\omega_3=\epsilon_1+\epsilon_2+\epsilon_3$ and  $$\omega_4=\frac{1}{2}(\epsilon_1+\epsilon_2+\epsilon_3+\epsilon_4-\epsilon_5), \qquad
\omega_5=\frac{1}{2}(\epsilon_1+\epsilon_2+\epsilon_3+\epsilon_4+\epsilon_5).$$
The sum of the fundamental weights is 
$$\rho=\omega_1+\omega_2+\omega_3+\omega_4+\omega_5=
4\epsilon_1+3\epsilon_2+2\epsilon_3+\epsilon_4.$$
The weights of $U^\vee$ are the $\epsilon_i$'s, so the weights of $\wedge^3U^\vee$ are the sums 
of three distinct $\epsilon_i$'s, and the highest one
is $\epsilon_1+\epsilon_2+\epsilon_3$. Thus the highest weight of $\wedge^3U^\vee(-2)$ is 
$\omega_3-2\omega_5$. Bott's theorems states that in order to find the cohomology groups of 
our bundle, we first need to add $\rho$ to this weight, which gives $\tau=\omega_1+\omega_2+2\omega_3+\omega_4-\omega_5$. No root of $D_5$ is orthogonal to $\tau$,
so there will be one non zero cohomology group. To find it, we choose a simple root with negative
scalar product with $\tau$; there is only one, $\alpha_5$; so we apply the associated simple 
reflection, which yields $s_5(\tau)=\tau+\alpha_5=\omega_1+\omega_2+\omega_3+\omega_4+\omega_5$. 
Since no coefficient is negative we do not need to repeat this operation; we just subtract $\rho$. 
This yields the weight zero, so we get a cohomology group isomorphic to the trivial representation
$\CC$, in degree one because we just needed to apply one simple reflection. 
(Note by  the way that $\wedge^3U^\vee(-2)\simeq\wedge^2U\simeq\Omega_\cS$, which explains why we obtain
$H^1(\cS,\wedge^3U^\vee(-2))=H^1(\cS,\Omega_\cS)=\CC$.)
\hfill $\Box$
\end{proof}

\medskip
The next step is to prove the following statement:

\begin{proposition} Suppose $X\subset\cS_1$, where $\cS_1$ is a translate of $\cS$. 
Let $N_1$ be the normal bundle to $\cS_1$ in $\PP\Delta$. Then from its restriction   
to $X$ one can reconstruct the embeddings $X\subset\cS_1\subset\PP(\Delta)$.
\end{proposition}

\begin{proof} Suppose $\cS_1=\cS$ to simplify the notations. 
Since $\cS$ is cut out by quadrics, our 
strategy will be to reconstruct its quadratic equations from $N_X$, or equivalently, from $U_X$. 
Then we simply recover $\cS$ as the zero locus of its quadratic equations.

 \smallskip\noindent {\bf Step 1}.
The key observation is that there is a natural isomorphism
\begin{equation}\label{iso}
H^0(\cS,U(1))\simeq H^0(\cS,\cO_\cS(1))^\vee.
\end{equation}
Indeed, $U^\vee$ is the irreducible bundle associated to the representation of highest weight $\omega_1=\epsilon_1$. 
The highest weight of its dual $U$ is $-\epsilon_5=\omega_4-\omega_5$. Therefore the highest weight of $U(1)$ is 
$\omega_4$, and the assertion follows from the Borel-Weil theorem.

 \smallskip\noindent {\bf Step 2}.
There are  natural morphisms 
$$H^0(\cS,U^\vee)\otimes H^0(\cS,U(1))\lra H^0(\cS,U^\vee\otimes U(1))\lra H^0(\cS,\cO_\cS(1)),$$
the right hand side being induced by the trace map $U^\vee\otimes U\ra \cO_\cS$. 
This determines the quadratic equations of $\cS$, as the image of the induced map
$$\begin{array}{ll}
V_{10}=V_{10}^\vee=H^0(\cS,U^\vee)\lra H^0(\cS,U(1))^\vee\otimes H^0(\cS,\cO_\cS(1))\simeq H^0(\cS,\cO_\cS(1))^{\otimes 2}\lra Sym^2H^0(\cS,\cO_\cS(1)).
 \end{array}$$
 (The composition is non zero because $\wedge^2H^0(\cS,\cO_\cS(1))$ does not contain any
 direct factor isomorphic to $V_{10}$.)
 
 \medskip With the help of these observations, we now need to prove that we can recover the 
 quadratic equations of $\cS$ just starting from $U_X$. We will show we can follow exactly 
 the same argument as above, using only spaces of sections of bundles defined on $X$ only.

 \smallskip\noindent {\bf Step 3}. We deduce from (\ref{resolutionX}) and the Borel-Weil-Bott theorem that the restriction morphism
 $$res_F : H^0(\cS,F)\lra H^0(X,F_X)$$
 is an isomorphism for either $F=\cO_\cS(1), U^\vee, U(1), U(2), Sym^2U^\vee$. 
 
 \smallskip\noindent {\bf Step 4}. 
 We recover the quadratic form (up to scalar) on $V_{10}^\vee=H^0(X,U_X^\vee)$ as a generator of the  one dimensional kernel of the map $Sym^2H^0(X,U_X^\vee)\lra H^0(X,Sym^2U_X^\vee)$;  hence the isomorphism $V_{10}^\vee\simeq V_{10}$. 
 
 \smallskip\noindent {\bf Step 5}. 
 We have $\Delta_+^\vee=H^0(X,\cO_X(1))$ and we want to identify its dual with $\Delta_-^\vee=H^0(X,U_X(1))$. 
 First note that we have a natural
 map  $\eta : V_{10}^\vee\otimes \Delta_-^\vee\lra\Delta^\vee$ obtained by multiplying
 sections and contracting:
 $$\eta: H^0(X,U_X^\vee)\otimes H^0(X,U_X(1))\lra H^0(X,U_X^\vee\otimes U_X(1))\lra 
 H^0(X,\cO_X(1)).$$
 
 \smallskip\noindent {\bf Step 6}. 
 From $\eta$ we get a map $\chi :  \Delta_-^\vee\lra\Delta^\vee\otimes V_{10}$, hence 
 also a map
 $$\xi : \Delta^\vee\otimes \Delta_-^\vee\lra\wedge^2\Delta^\vee\otimes V_{10}.$$ 
We claim that we can recover the duality between $\Delta$ and $\Delta_-$ as being given by
the one-dimensional kernel $K\subset \Delta^\vee\otimes \Delta_-^\vee$ of the $\xi$. 
Because of the results of Step 3, it is enough to check this claim when we consider the same maps as defined on 
$\cS$ rather than on $X$. We use the following decompositions into irreducible $\Spin_{10}$-modules
\cite{LiE}:
$$\Delta^\vee\otimes \Delta_-^\vee = V_{\omega_4+\omega_5}\oplus V_{\omega_2}\oplus\CC,$$
$$\wedge^2\Delta^\vee\otimes V_{10}=V_{\omega_4+\omega_5}\oplus V_{\omega_2}\oplus V_{\omega_1+\omega_3}.$$
There is a unique irreducible factor $K=\CC$ that appears in the first decomposition and 
not in the second one. Elementary computations, left to the reader, allow to check that $K$ is exactly the 
kernel of $\xi$. Note that since $\Delta$ and $\Delta_-$ are irreducible, $K$ must define a perfect 
duality between these modules. 

 \smallskip\noindent {\bf Step 7}.
 Once we have the duality defined by $K\subset \Delta^\vee\otimes \Delta_-^\vee$, 
 we can argue exactly as in Step 2 to realize $V_{10}$, just starting from $X$, as a system 
 of quadrics on $\PP (\Delta)$. We finally recover $\cS$ as the base locus of this system.~\hfill $\Box$ 
\end{proof}

\paragraph{Remark.} The key isomorphism (\ref{iso}) can be explained as follows. 
Recall that $\cS\subset\PP\Delta$ and $\cS^\vee\subset\PP\Delta^\vee$ parametrize the 
two families of maximal isotropic spaces of $V_{10}$. Two such spaces belong to different
families if and only if they intersect in even dimension. Moreover, any isotropic four-plane
is contained in exactly two maximal isotropic subspaces, one from each family. We can therefore
identify the orthogonal Grassmannian $OG(4,V_{10})$ with 
the incidence variety $\cI\subset \cS\times \cS^\vee $ of pairs $(U,U')$ such that $L=U\cap U'$ 
has dimension four, and denote by $p, p_\vee$ the two projections. By the preceding observations,
the projection $p$ identifies $\cI$ with the projective bundle $\PP U^\vee$. Moreover, with the 
previous notations, we have $\det(U)=\det(L)\otimes U/L$, $\det(U')=\det(L)\otimes U'/L$, and 
the quadratic form induces a natural duality between $U/L$ and $U'/L$. From this one easily 
deduces that $p_\vee^*\cO_{\cS^\vee}(1)= \cO_{\PP U^\vee}(1)\otimes p^*\cO_\cS(1)$. And then
$$H^0(\cS,\cO_\cS(1))^\vee = H^0(\cS^\vee,\cO_{\cS^\vee}(1)) = H^0(\cI,p_\vee^*\cO_{\cS^\vee}(1))=H^0(\cI,  \cO_{\PP U^\vee}(1)\otimes p^*\cO_\cS(1))=H^0(\cS,U(1)).$$

\medskip
By the same argument as in \cite[Proposition 2.3]{bcp}, we deduce from the previous Proposition that:

\begin{proposition} 
If $X=\cS_1\cap\cS_2$ is a transverse intersection, then the only translates of $\cS$ 
that contain $X$ are   $\cS_1$ and  $\cS_2$. 
\end{proposition}

\section{Double mirrors}

Recall that the spinor variety $\cS\subset\PP\Delta$ is projectively dual to the other spinor variety 
$\cS^\vee\subset\PP\Delta^\vee$. We may therefore associate to $X=\cS_1\cap \cS_2\subset\PP\Delta$, the other double spinor variety $Y=\cS_1^\vee\cap \cS_2^\vee\subset\PP\Delta^\vee$. 
When $X$ is smooth, its presentation as the intersection of two translated spinor varieties is unique, and therefore 
 $Y$ is uniquely defined.

\subsection{Derived equivalence}

\begin{proposition}
The double spinor varieties $X$ and $Y$ are simultaneously smooth of expected dimension.
\end{proposition}

\begin{proof} 
Suppose $\cS_1=g_1\cS$ and $\cS_2=g_2\cS$ and let $x\in X$. Then $x=g_1E_1=g_2E_2$ for some $E_1,E_2$ in $\cS$. 
The intersection of $\cS_1$ and $\cS_2$ fails to be transverse at $x$ if and only if there is a point $y\in\PP\Delta^\vee$ such that the corresponding hyperplane $H_y$ in $\PP\Delta$ is tangent to both $\cS_1$ and $\cS_2$ at $x$. By Lemma \ref{tangency}, this means that 
 $y=g_1^tF_1=g_2^tF_2$ for some $F_1,F_2$ in $\cS^\vee$, such that 
 $\dim (E_1\cap F_1) = \dim (E_2\cap F_2) = 4$. In particular $y$ belongs to $g_1^t\cS^\vee\cap g_2^t\cS^\vee = \cS_1^\vee\cap\cS_2^\vee =Y$ and by 
 symmetry, the intersection of $\cS_1^\vee$ and $\cS_2^\vee$ fails to be transverse at $y$. This implies the claim. 
\hfill $\Box$ 
\end{proof}

\begin{proposition}\label{dereq}
When they are smooth, the double spinor varieties $X$ and $Y$ are derived equivalent.
\end{proposition}

From the point of view of mirror symmetry, $X$ and $Y$ being D-equivalent should have the same mirror: they form an instance of a double mirror. 

\begin{proof} This is a direct application of the results of \cite{kp}, or of 
the Main Theorem in \cite{jlx}. As we already mentioned, the
fact that $(\cS_1, \cS_1^\vee)$ and $(\cS_2,\cS_2^\vee)$ are pairs of homologically projectively dual varieties was established in \cite[Section 6.2]{kuz}.  
\hfill $\Box$
\end{proof}
 
\paragraph{Remark.} In fact the results of \cite{jlx, kp} imply the stronger 
statement that $X$ and $Y$ are derived equivalent as soon as they have dimension five, even
if they are singular. Since the smoothness of a variety can be detected at the level of 
its derived category, this provides another proof of Proposition 4.1. 

\medskip

Applying Proposition 2.1 of \cite{or}, we deduce (recall from Proposition \ref{cy} that $H^5(X, \ZZ)$ and $H^5(Y, \ZZ)$ are torsion free):

\begin{corollary}\label{hodgeq}
The polarized Hodge structures on $H^5(X, \ZZ)$ and $H^5(Y, \ZZ)$ are equivalent. 
\end{corollary}

\subsection{Non birationality} 

Now we sketch a proof of the following result, according to the ideas of \cite[Proof of Lemma 4.7]{or}. 

\begin{proposition}\label{nbireq}
Generically, the mirror double spinors  $X$ and $Y$ are not birationally equivalent.
\end{proposition}

\begin{proof} 
By a standard argument, it is enough to prove that $X$ and $Y$ in $\PP(\Delta)$ are not projectively equivalent. Indeed, suppose $X$ and $Y$ are birational. Since they are Calabi-Yau, $X$ and $Y$ are
minimal models, and this implies that the birational equivalence must be an isomorphism in 
codimension two. Since their Picard groups are both cyclic, the birational equivalence identifies 
their (very ample) generators, and induces an isomorphism between their spaces of sections, 
yielding a projective equivalence as claimed. 

So suppose that $X$ and $Y$ are projectively equivalent. Since they are both contained in a unique pair
of translates $\cS_1$ and $\cS_2$  of the spinor variety, there would exist a projective isomorphism
$u:\PP(\Delta)\simeq \PP(\Delta^\vee)$ such that either $u(\cS_1)=\cS_1^\vee$ and $u(\cS_2)=\cS_2^\vee$, or $u(\cS_1)=\cS_2^\vee$ and $u(\cS_2)=\cS_1^\vee$. 

Let us fix once and for all a linear isomorphism $u_0:\PP(\Delta)\simeq \PP(\Delta^\vee)$ such that $u_0(\cS_1)=\cS_1^\vee$.
There is a linear automorphism $g$ of $\PP(\Delta)$ such that $\cS_1=g(\cS_2)$. 
It is easy to check that the existence of $u$ is equivalent to the existence of $v, w$ in $\Aut(\cS_1)$ such that either 
$$u_0g^tu_0^{-1}=vg^{-1}w \qquad {\mathrm or}\qquad u_0g^tu_0^{-1}=vgw.$$
We follow the approach of \cite{or} to prove that for a general $g$, such elements of $H=\Aut(\cS_1)$ do not exist.

\medskip\noindent {\bf First case}. In order to exclude the possibility that $u_0g^tu_0^{-1}=vg^{-1}w$, one might 
exhibit an $H\times H$-invariant function on $G=\PGL(\Delta)$ such that $F(g^{-1})\ne F(u_0g^tu_0^{-1})$. To do this, recall that
the quadratic equations of the spinor variety $\cS\subset\PP(\Delta)$ are parametrized by $V_{10}\simeq V_{10}^\vee\subset S^2\Delta^\vee$. 
The invariant quadratic form $q\in S^2V^\vee$ is thus mapped to an invariant element $Q\in S^2\Delta^\vee\otimes S^2\Delta^\vee$. 
(In fact this element belongs to the kernel of the product map to $S^4\Delta^\vee$, since the latter contains no invariant.)
Dually there is an invariant element $Q^\vee\in S^2\Delta\otimes S^2\Delta$, and the function we use is $F(g)=\langle Q^\vee, gQ\rangle $. (This is actually a function on $\SL(\Delta)$, but a suitable
power will descend to $G=\PGL(\Delta)$.)  Indeed, restricted to a maximal torus of $\SL(\Delta)$, 
$F(u_0g^tu_0^{-1})$ is a polynomial function of degree four, and it cannot coincide with $F(g^{-1})$ which is a polynomial of degree four in the inverses of the variables -- even modulo the condition that the 
product of the sixteen variables is one. 

\medskip\noindent {\bf Second case}. As observed in \cite[Proof of Lemma 4.7]{or}, it suffices to show that there exists some partition 
$\lambda$ such that the space of $H$-invariants in $S_\lambda\Delta$ is at least two-dimensional ($S_\lambda$ denotes the
Schur functor associated to the partition $\lambda$). We provide an 
abstract argument for that. Suppose the contrary. Let 
$G=\SL(\Delta)$. By the Peter-Weyl theorem, the multiplicity of $S_\lambda\Delta$ inside $\CC[G/H]$ 
is the dimension of its subspace of $H$-invariants. If this dimension is always smaller or equal to one, 
$\CC[G/H]$ is multiplicity free, which means that $H$ is a {\it spherical subgroup} of $G$. Then by \cite[Theorem 1]{vk}, 
$H$ has an open orbit in the complete flag variety $Fl(\Delta)$. But the dimension of $H$ is just too small for that 
to be true, and we get a contradiction.
\hfill $\Box$ 
\end{proof}

\subsection{L-equivalence}

Recall that ${\mathbf L}$ denotes the class of the affine line in the Grothendieck ring of complex varieties. 

\begin{proposition}\label{relgro}
The double spinor varieties $X$ and $Y$ are such that
$$([X]-[Y]){\mathbf L}^7=0$$
in the Grothendieck ring of varieties. 
\end{proposition}

Note that when $X$ and $Y$ are not birational, $[X]-[Y]\ne 0$ in the Grothendieck ring (see \cite[Proposition 2.2]{ks}). 

\begin{proof} The proof is the same as for Theorem 1.6 in \cite{bcp}. 
We consider the incidence correspondence 
$$\xymatrix{
 & Q  \ar@{->}[rd]^{p_2}   \ar@{->}[ld]_{p_1} \\ \cS_1 & & \cS_2^\vee
 }$$
where $Q$ is the variety of pairs $x\in \cS_1$, $y\in \cS_2^\vee$ such that $x$ belongs to the 
 hyperplane $H_y$. The fiber of $p_2$ over $y$ is $\cS_1\cap H_y$; it is singular if and only if 
 $y$ also belongs to $\cS_1^\vee$, hence to $Y$. In this case the fiber is isomorphic
 to $\cHS_{sing}$, otherwise it is isomorphic to $\cHS_{reg}$. This yields two fibrations with constant 
 fibers, which may not be Zariski locally trivial but must be {\it piecewise trivial}, like in  
 \cite[Lemme 3.3]{martin}. Indeed, by \cite[Theorem 4.2.3]{sebag}, this follows from the already
 mentioned result of Igusa that over any field (not of characteristic two) over which the 
 spinor group splits, in particular over any field containing $\CC$, there are only two orbits 
 of non zero spinors \cite[Proposition 2]{igusa}.
 
 We deduce that in the  Grothendieck ring of varieties, 
 $$ [Q]  =  [Y][\cHS_{sing}]+[\cS_2^\vee-Y][\cHS_{reg}].$$
The same analysis for the other projection yields the symmetric relation
$$ [Q]  =  [X][\cHS_{sing}]+[\cS_1-X][\cHS_{reg}].$$
Taking the difference (recall that $\cS_1$ and $\cS_2^\vee$ are isomorphic varieties), we get
$$0=([X]-[Y])([\cHS_{sing}]-[\cHS_{reg}]).$$
But $\cHS_{sing}$ and $\cHS_{reg}$ both have cell decompositions (Propositions \ref{hssing} and \ref{hsreg}), 
with the same numbers of
cells except that $\cHS_{reg}$ has one less in dimension seven. Hence $[\cHS_{sing}]-[\cHS_{reg}]={\mathbf L}^7$.
\hfill $\Box$ 
\end{proof}

\bibliographymark{References}

\end{document}